\documentclass[11pt]{amsart}
\usepackage[colorlinks=true,linkcolor=blue,hypertexnames=false]{hyperref}

\usepackage{enumerate}
\usepackage{lineno}
\usepackage{graphicx}
\usepackage{geometry}
\usepackage{enumitem}
\makeatletter
\@namedef{subjclassname@}{%
\textup{}}
\makeatother
\newtheorem{thm}{Theorem}[section]
\newtheorem{thmx}{Theorem}[section]

\newtheorem{case}{Case}
\newtheorem{subcase}{Subcase}

\newtheorem*{conj*}{Denjoy's Conjecture}

\theoremstyle{definition}

\newtheorem{rem}{Remark}[section]
\newtheorem{Que}{Question}
\newtheorem{exa}{Example}[section]
\numberwithin{equation}{section}

\begin{document}

\title[]{ On entire  solutions of two kinds of   Fermat type functional equations }

\author{Jiayi Dan\quad Xuxu Xiang*}

\address{Jiayi Dan \newline School of Mathematical Sciences, Guizhou Normal University, Guiyang, 550025, P.R. China. }
\email{984455342@qq.com}

\address{Xuxu Xiang \newline School of Mathematical Sciences, Guizhou Normal University, Guiyang, 550025, P.R. China. }
\email{1245410002@qq.com}


\date{}


\begin{abstract}The existence of entire solutions to Fermat type differential-difference equations and\\
 \(q\)-difference differential equations involving second-order derivatives is investigated by using Nevanlinna theory. The exact forms of the entire solutions to these equations mentioned above are identified. These results represent a generalization and improvement of previous findings obtained by Gong et al. [Bull. Iran. Math. Soc. 51, 17 (2025)]. Furthermore, some examples are provided to illustrate these results.

\end{abstract}



\keywords{Entire  solutions; differential-difference equations; \(q\)-difference differential equations; existence;  Fermat type functional equations
*Corresponding author. \\
}
\maketitle
\section{Introduction }\label{sec1}

~~~~
Let $f$ be a meromorphic function in the complex plane $\mathbb{C}$. Assume that the reader is familiar with the standard notation and basic results of Nevanlinna theory, such as $m(r,f),~N(r,f)$,$~T(r,f)$, see \cite{hayman} for more details. A meromorphic
function $g$ is said to be a small function of $f$ if $T(r,g)=S(r,f)$, where $S(r,f)$  denotes any quantity
that satisfies $S(r,f)= o(T(r, f))$ as $r$ tends to infinity, outside of a possible exceptional set of finite linear measure. $\rho(f)$   denotes the order of $f$. The hyper-order of $f$ is defined by  $\rho_2(f)=\underset{r\rightarrow \infty}{\lim\sup}\frac{\log^+\log^+T(r,f)}{\log r}.$  In order to conveniently introduce some results, we need the following notations.

\begin{align}
\label{eq1}
\mathcal{A} &= \begin{pmatrix}
a_0 & a_1 & a_2 \\
b_0 & b_1 & b_2
\end{pmatrix}, ~\mathcal{A}_1=a_0b_1-a_1b_0,~\mathcal{A}_2=a_1b_2-a_2b_1,~\mathcal{A}_3=a_0b_2-a_2b_0,\\ \nonumber
    M_1(f) &= a_0 f(z) + a_1 f'(z) + a_2 f(z + c),~ M_2(f) = b_0 f(z) + b_1 f'(z) + b_2 f(z + c),\\ \nonumber
     L_1(f) &= a_0 f(z) + a_1 f''(z) + a_2 f(z + c),~ L_2(f) = b_0 f(z) + b_1 f''(z) + b_2 f(z + c),
\end{align}
  where \( a_i, b_i, i = 0, 1, 2 \), and \( c (\neq 0) \) are constants. And, the rank of $\mathcal{A}$ is denoted by \( r(\mathcal{A}) \).

In 1970, Yang \cite{Yang1970} applied Nevanlinna theory to prove that if $\frac{1}{m}+\frac{1}{n}<\frac{2}{3}$, then the Fermat type functional equation
\begin{align}
\label{0.1.2}
a(z)h^n(z) + b(z)g^m(z) = 1
\end{align}
does not admit non-constant meromorphic solutions $h(z)$ and $g(z)$, where $a(z)$ and $b(z)$ are small functions with respect to $h(z)$ and $g(z)$, and $m, n$ are positive integers.
Yang's results extend the findings of   Montel \cite{Montel}. Some important results related to Fermat type functional equations can be found in \cite{B,Chen,Gross1,Gross2}.
According to the case of $m = n = 2$, $a = b = 1$ and there is a connection between $h$ and $g$ in \eqref{0.1.2}, scholars have conducted extensive research on the entire solution of \eqref{0.1.2}. For example, Yang et al. \cite{yang2004} considered the case $g = h'$; Liu et al. \cite{Lc} studied the case $g = h(z + c)$; Hu et al. \cite{Hu} explored the case $h = M_1(f)$ and $g = M_2(f)$, where $M_1(f), M_2(f)$ are defined in \eqref{eq1}.

We next recall some results on the existence of solutions of  trinomial Fermat type functional equation
\begin{align}
\label{0.1.3}
f^m+h^p+g^n=1,
\end{align}
where $f,g,h$ are meromorphic functions, $m, n, p $ are positive integers.
In 1985,  Hayman \cite{h1} proved that if $n=m=p\ge9$, then there do not exist distinct transcendental meromorphic
 functions $f,g$ and $h$ that satisfy \eqref{0.1.3}. Later, Gundersen \cite{G2,G1}  gave examples of transcendental meromorphic functions $f, g, h$ satisfying \eqref{0.1.3} for the cases
$n=m=p = 6, 5$, the cases $n = 7$ and $n = 8$ are still open. For the cases $n=m=p\le5$, it is known that there exist transcendental
entire functions $f, g, h$ satisfying \eqref{0.1.3}, see \cite{G2,G3}.
 In 1971, Toda \cite{Toda} proved that  there do not exist non-constant entire functions $f, g,h $ satisfying \eqref{0.1.3} for $n=m=p\ge7$,  the case $n = 6$ is still open.

 In 2008, Yang et al. \cite{Yl} proved that there do not exist non-constant meromorphic functions $f, g,h $ satisfying \eqref{0.1.3} for $\frac{1}{n}+\frac{1}{m}+\frac{1}{p}<\frac{1}{3}$. For the case $m=2,n=2,p=1$, and $h=2afg$ in \eqref{0.1.3},  Saleeby \cite{S}  got   the entire solutions of  quadratic trinomial Fermat type equation
 \begin{align}
\label{1.2}
f^2+2afg+g^2=1 ,\,\, a^2\not=1,\, a\in \mathbb C,
 \end{align}
are
$f(z)=\frac{1}{\sqrt{2}}\left(\frac{\cos(b)}{\sqrt{1+a}}+\frac{\sin(b)}{\sqrt{1-a}}\right) ,g(z)=\frac{1}{\sqrt{2}}\left(\frac{\cos(b)}{\sqrt{1+a}}-\frac{\sin(b)}{\sqrt{1-a}}\right),$
where $b$ is any entire function. Later, Liu et al. \cite{Liu} proved that if $g(z)=f'(z)$ in \eqref{1.2}, then the equation
\begin{align}
\label{1.3}
f^2+2aff'+f'^2=1 ,\,\, a^2\not=1,0,\, a\in \mathbb C
\end{align}
 has no transcendental meromorphic solutions. At the same time, they also proved if $g(z)=f(z+c)$ in \eqref{1.2}, then the finite order transcendental entire
function solutions of equation
\begin{align}
\label{1.4}
f(z)^2+2af(z)f(z+c)+f^2(z+c)=1 ,\,\, a^2\not=1,0,\, a\in \mathbb C
\end{align}
must be of order equal to one. Recently, there are many explorations concerning the quadratic trinomial Fermat type equations, see \cite{Gong,Haldar,tai,Wang,Xu} and references
therein.

In 2020, Wang et al. \cite{Wang} considered the more general forms of \eqref{1.3} and \eqref{1.4}. They obtained the following conclusions.

\begin{thmx}\cite{Wang}
	\label{thA}
	Assume that  $M_1, M_2$ are defined in \eqref{eq1}, \(\alpha, \beta, \omega (\neq 0, \pm 1)\) are complex numbers, and \(r(\mathcal{A}) = 2\). If the equation
\begin{equation}
\label{1.7}
(M_1(f))^2 + 2\omega M_1(f)M_2(f) + (M_2(f))^2 = e^{\alpha z+\beta}
\end{equation}
admits an entire solution \(f\) of finite order, then \(f\) must take the form
\begin{equation*}
f(z)=C_1 e^{(\alpha/2 + A)z}+C_2 e^{(\alpha/2 - A)z}+C_3 e^{C_0 z},
\end{equation*}
where \(A (\neq 0), C_i, i = 0,1,2,3\), are constants satisfying \(C_0 = C_3 = 0\) if \(\mathcal{A}_2 = 0\), and \(C_0 = -\dfrac{\mathcal{A}_2}{\mathcal{A}_2}\) if \(\mathcal{A}_2 \neq 0\).
\end{thmx}

\begin{rem}
For simplicity, we do not add the expressions of the constants
$C_i,i=1,2,3,$ to Theorem \ref{thA}. The readers can refer to the original paper for these  expressions.  However, Liu et al. \cite[Example 1, 2]{Liu1} pointed out that that there are some entire solutions which were missing in Theorem \ref{thA}.  Liu et al. \cite[Theorem 1]{Liu1} also  given   all the  entire solutions  of \eqref{1.7} under the
weakened growth condition when $\alpha z+\beta$ is replaced by a non-constant polynomial.
\end{rem}

Later, Gong et al. \cite{Gong} extended the first-order derivatives $f'$ in \eqref{1.7} to second-order derivatives $f''$. Upon carefully reviewing  Gong et al. \cite[Theorem 1.5]{Gong}, we can summarize and express their result as follows.

\begin{thmx}\cite[Theorem 1.5]{Gong}
	\label{thB}
	Assume that  $L_1, L_2$ are defined in \eqref{eq1}, \(\alpha, \beta, \omega (\neq 0, \pm 1)\) are complex numbers, and \(r(\mathcal{A}) = 2\). If the equation
\begin{equation}
\label{1.8}
(L_1(f))^2 + 2\omega L_1(f)L_2(f) + (L_2(f))^2 = e^{\alpha z+\beta}
\end{equation}
admits an entire solution \(f\) of finite order, then one the following
asserts holds.
	\begin{itemize}
		\item[\rm{(i)}]If \(\mathcal{A}_2= 0\), then $f(z) = C_1 e^{(A + \frac{\alpha}{2})z} + C_2 e^{(-A + \frac{\alpha}{2})z}$, where $ A, B$ are arbitrary constants,
        \(C_1=\frac{(b_2\omega_2 - a_2)e^{B+\frac{\beta}{2}}}{(\omega_2 - \omega_1)\mathcal{A}_3}\) and  \(C_2=\frac{(a_2 - b_2\omega_1)e^{-B+\frac{\beta}{2}}}{(\omega_2 - \omega_1)\mathcal{A}_3}\) or \(C_1=\frac{(b_1\omega_2 - a_1)e^{B+\frac{\beta}{2}}}{(\omega_2 - \omega_1)\mathcal{A}_1}\) and \(C_2=\frac{(a_1 - b_1\omega_1)e^{-B+\frac{\beta}{2}}}{(\omega_2 - \omega_1)\mathcal{A}_1}\).
		\item[\rm{(ii)}]If  \(\mathcal{A}_2\not= 0\), and \(\rho(f)\leq 1\), then
     $  f(z) = C_1 e^{(A + \frac{\alpha}{2})z} + C_2 e^{(-A + \frac{\alpha}{2})z} + C_3 e^{C_5 z} + C_4 e^{C_6 z}$, where $ A, B, C_3,C_4$ are arbitrary constants, and $C_1, C_2, C_5,C_6$ depend on $a_i,i=1,2,3,$ and $\alpha$  as follows:
     \begin{itemize}
         \item[\rm{(1)}]If \((\pm A+\frac{\alpha}{2})^2\neq\frac{-\mathcal{A}_3}{\mathcal{A}_2}\), then $
    C_5=\sqrt{-\frac{\mathcal{A}_3}{\mathcal{A}_2}}, C_6=-\sqrt{-\frac{\mathcal{A}_3}{\mathcal{A}_2}} $,
   $ C_1=\frac{4(b_2\omega_2 - a_2)e^{\frac{\beta + 2B}{2}}}{(\omega_2-\omega_1)[\mathcal{A}_2(2A+\alpha)^2+4\mathcal{A}_3]},$\\
   $ C_2=\frac{4(a_2 - b_2\omega_1)e^{\frac{\beta - 2B}{2}}}{(\omega_2-\omega_1)[\mathcal{A}_2(-2A+\alpha)^2+4\mathcal{A}_3]};$
         \item[\rm{(2)}]If \((A+\frac{\alpha}{2})^2=\frac{-\mathcal{A}_3}{\mathcal{A}_2}\) and \(A\neq0\), then $ C_5 = A+\frac{\alpha}{2}, C_6=-A-\frac{\alpha}{2},$
    $C_1=\frac{(a_2 - b_2\omega_2)e^{\frac{\beta + 2B}{2}}z}{(2A+\alpha)(\omega_2-\omega_1)\mathcal{A}_2},$\\
    $C_2=\frac{(a_2 - b_2\omega_1)e^{\frac{\beta - 2B}{2}}}{2\alpha A(\omega_2-\omega_1)\mathcal{A}_2};$
         \item[\rm{(3)}]If \((-A+\frac{\alpha}{2})^2=\frac{-\mathcal{A}_3}{\mathcal{A}_2}\) and \(A\neq0\), then   $ C_5=-A+\frac{\alpha}{2}, C_6=A-\frac{\alpha}{2}$,
   $  C_1=\frac{(b_2\omega_2 - a_2)e^{\frac{\beta + 2B}{2}}}{-2\alpha A(\omega_2-\omega_1)\mathcal{A}_2}, $\\
   $C_2=\frac{(a_2 - b_2\omega_1)e^{\frac{\beta - 2B}{2}}z}{(2A-\alpha)(\omega_2-\omega_1)\mathcal{A}_2};$

         \item[\rm{(4)}]
         If \(\frac{\alpha^2}{4}\mathcal{A}_2=-\mathcal{A}_3\), then $C_5 = C_6=\frac{\alpha}{2}, $
         $ C_1=\frac{(b_2\omega_2 - a_2)e^{\frac{\beta + 2B}{2}}z}{\alpha(\omega_2-\omega_1)\mathcal{A}_2},$
         $C_2=\frac{(a_2 - b_2\omega_1)e^{\frac{\beta - 2B}{2}}z}{\alpha(\omega_2-\omega_1)\mathcal{A}_2};$
     \end{itemize}
      $ \text{where } \omega_1=-\omega+\sqrt{\omega^2 - 1}, \omega_2=-\omega-\sqrt{\omega^2 - 1}.$
	\end{itemize}
\end{thmx}

By replacing $L_1$ with $Q_1(f)$ and $L_2(f)$ with $Q_2(f)$ in \eqref{1.8}, respectively,  where
\begin{align}
\label{Q}
Q_1(f)=a_0 f(z) + a_1 f''(z) + a_2 f(qz), Q_2(f) = b_0 f(z) + b_1 f''(z) + b_2 f(qz),
\end{align}
 Gong et al. \cite{Gong} obtained the $q$-difference analogue of Theorem \ref{thB}.

\begin{thmx}\cite[Theorem 1.7]{Gong}
	\label{thC}
	Assume that  \(\alpha, \beta, \omega (\neq 0, \pm 1),q\) are complex numbers such that \(q^n\neq 0, \pm 1,\) $n\in \mathbb{N}^+$, and \(r(\mathcal{A}) = 2\). If the equation
\begin{equation}
\label{1.9}
(Q_1(f))^2 + 2\omega Q_1(f)Q_2(f) + (Q_2(f))^2 = e^{\alpha z+\beta}
\end{equation}
admits an entire solution \(f\) of $\rho(f)<\infty$, then one the following
asserts holds.
	\begin{itemize}
		\item[\rm{(i)}]If \(\mathcal{A}_2= 0\), then the conclusion  of Theorem \ref{thB}-$\rm{(i)}$ holds.
		\item[\rm{(ii)}]If  \(\mathcal{A}_2\not= 0\) and \(\rho(f)\leq 1\), then
     $  f(z) = C_1 e^{(A + \frac{\alpha}{2})z} + C_2 e^{(-A + \frac{\alpha}{2})z} + C_3 e^{C_5 z} + C_4 e^{C_6 z}$ where $ A, B, C_3,C_4$ are arbitrary constants, and $C_1, C_2, C_5,C_6$ are given as same as the conclusion in  $(1)-(3)$ of Theorem \ref{thB}-$\rm{(ii)}$ .
	\end{itemize}
\end{thmx}

Inspired by  Theorem \ref{thB}  and  Theorem \ref{thC}, it is natural to ask the following question:
\begin{Que}
\label{Q1}
Can the condition \(\rho(f)\leq 1\)  be removed from assertion
$\rm{(ii)}$ of Theorem \ref{thB} and Theorem \ref{thC},  respectively ?
\end{Que}

The following example illustrates that when $\alpha z+\beta$ in \eqref{1.8}  is replaced by a polynomial $g(z)=z^{2}$, there exists an entire function solution satisfies the equation \eqref{1.8}.
\begin{exa}\cite[Example 3]{Liu1}
\label{ex1}Let $\omega = 2$, $g(z) = z^2$, and
    $L_1(f) = -(2 + \sqrt{3})f(z) + (2 - \sqrt{3})f(z + c),
    L_2(f) = f(z) - f(z + c),$
where $c$ is a non-zero constant. Then \eqref{1.8} admits an entire solution
   $ f(z) = -\frac{\sqrt{3}}{6} e^{z^2 - cz}.$

\end{exa}

Inspired by  Example \ref{eq1}, it is interesting to ask the following question:
\begin{Que}
\label{Q2}
Can we find out all possible forms of entire solutions when $\alpha z+\beta$ in \eqref{1.8} and \eqref{1.9} is replaced by a  polynomial $g(z)$ with $\deg g >1$?
\end{Que}	

The purpose of this article is to solve Question \ref{Q1} and Question \ref{Q2}. In  Sections \ref{s2} and \ref{s3}, we study all possible forms of entire solutions \(f\)  for the quadratic trinomial Fermat  type differential-difference equation \((L_1(f))^2 + 2\omega L_1(f)L_2(f) + (L_2(f))^2 = e^{g(z)}\) and the \(q\)-difference differential equation \((Q_1(f))^2 + 2\omega Q_1(f)Q_2(f) + (Q_2(f))^2 = e^{g(z)}\),
respectively. Some examples are provided in Section 4 to illustrate the results from Sections 2 and 3.

\section{quadratic trinomial Fermat  type differential-difference equations }
\label{s2}
In this section, we study the existence of solutions to the Fermat  type differential-difference equation \((L_1(f))^2 + 2\omega L_1(f)L_2(f) + (L_2(f))^2 = e^{g(z)}\). The parts related to Theorem \ref{thB} in Questions \ref{Q1} and \ref{Q2} are solved here.

\begin{thm}
	\label{th1}
	Assume that  $L_1, L_2$ are defined in \eqref{eq1}, \( \omega (\neq \pm 1)\) is a non-zero complex numbers, $\omega_1=-\omega+\sqrt{\omega^2 - 1}, \omega_2=-\omega-\sqrt{\omega^2 - 1},$ $g$ is a non-constant polynomial, and \(r(\mathcal{A}) = 2\). If the equation
\begin{equation}
\label{2.1}
(L_1(f))^2 + 2\omega L_1(f)L_2(f) + (L_2(f))^2 = e^{2g(z)}
\end{equation}
admits an entire solution \(f\) with  $\rho_2(f)<1$, then one the following
cases holds.
	\begin{itemize}
		\item[\rm{(1)}] $ g=\alpha z+\beta$, \(\mathcal{A}_2= 0\), $f(z) = C_1 e^{(A + \alpha)z} + C_2 e^{(-A + \alpha) z}$,
         where   $\alpha(\not=0),\beta, A, B$ are  constants, \(C_1=\frac{(b_2\omega_2 - a_2)e^{B+\beta}}{(\omega_2 - \omega_1)\mathcal{A}_3}\) and  \(C_2=\frac{(a_2 - b_2\omega_1)e^{-B+\beta}}{(\omega_2 - \omega_1)\mathcal{A}_3}\) or \(C_1=\frac{(b_1\omega_2 - a_1)e^{B+\beta}}{(\omega_2 - \omega_1)\mathcal{A}_1}\) and \(C_2=\frac{(a_1 - b_1\omega_1)e^{-B+\beta}}{(\omega_2 - \omega_1)\mathcal{A}_1}\).

\item[\rm{(2)}]$ g=\alpha z+\beta$, \(\mathcal{A}_2\not= 0\) and \(\mathcal{A}_3= 0\),  \begin{align*}
   f(z) = \frac{1}{\mathcal{A}_2 (\omega_2 - \omega_1)}\cdot
\begin{cases}
\left( A_1 e^{(\alpha + A)z} + A_2 e^{(\alpha - A)z} \right) + C_3 z + C_4, & \alpha \neq \pm A ,\\
\left( D_1 e^{2A z} + D_2 z^2 \right) + C_3 z + C_4, & \alpha = A, \\
\left( C_1 z^2 + C_2 e^{-2A z} \right) + C_3 z + C_4, & \alpha = -A,
\end{cases}
 \end{align*}
where $\alpha(\not=0), \beta,A,B,C_3,C_4$ are constants and

   $ A_1 = \dfrac{(b_2 \omega_2 - a_2) e^{\beta + B}}{(\alpha + A)^2} ,A_2 = \dfrac{(a_2 - b_2 \omega_1) e^{\beta - B}}{(\alpha - A)^2},$ $
D_1 = \dfrac{(b_2 \omega_2 - a_2) e^{\beta + B}}{4A^2},$\\
$D_2 = \dfrac{(a_2 - b_2 \omega_1) e^{\beta - B}}{2},
C_1 = \dfrac{(b_2 \omega_2 - a_2) e^{\beta + B}}{2},C_2 = \dfrac{(a_2 - b_2 \omega_1) e^{\beta - B}}{4A^2}$.

    \item [\rm{(3)}]$ g=\alpha z+\beta$,  \(\mathcal{A}_2\not= 0\) and \(\mathcal{A}_3\not= 0\),
\begin{align*}
f(z) = C_1 e^{\mu z} + C_2 e^{-\mu z} +
\begin{cases}
P_1 e^{k_1 z} + P_2 e^{k_2 z}, & k_1 \neq \pm \mu, k_2 \neq \pm \mu, \\
Q_1 z e^{k_1 z} + P_2 e^{k_2 z}, & k_1 = \pm \mu, k_2 \neq \pm \mu, \\
P_1 e^{k_1 z} + Q_2 z e^{k_2 z}, & k_1 \neq \pm \mu, k_2 = \pm \mu,
\end{cases}
\end{align*}
where  $\alpha(\not=0), \beta,C_1, C_2, A,B$ are constants, $\mu = \sqrt{-\dfrac{\mathcal{A}_3}{\mathcal{A}_2}},k_1 = \alpha + A,k_2 = \alpha - A $,
\begin{align*}
&P_1 = \dfrac{(b_2 \omega_2 - a_2) e^{\beta + B}}{(\omega_2 - \omega_1)[\mathcal{A}_2 (\alpha + A)^2 + \mathcal{A}_3]} ,P_2 = \dfrac{(a_2 - b_2 \omega_1) e^{\beta - B}}{(\omega_2 - \omega_1)[\mathcal{A}_2 (\alpha - A)^2 + \mathcal{A}_3]}, \\
&Q_1 = \dfrac{(b_2 \omega_2 - a_2) e^{\beta + B}}{2\mathcal{A}_2 (\omega_2 - \omega_1)(\alpha + A)},Q_2 = \dfrac{(a_2 - b_2 \omega_1) e^{\beta - B}}{2\mathcal{A}_2 (\omega_2 - \omega_1)(\alpha - A)}.
\end{align*}
		\item[\rm{(4)}] $\deg g>1$, \(\mathcal{A}_2 = \mathcal{A}_1= 0\), $ f(z)=\frac{ (a_2 - b_2\omega_1)e^{g - p}}{\mathcal{A}_3(\omega_2 - \omega_1)}$ or $ f(z)=\frac{ (b_2\omega_2 - a_2)e^{g(z)+p(z)}}{\mathcal{A}_3(\omega_2 - \omega_1)}$ , where $p$ is a polynomial of $\deg p=\deg g-1$.
	\end{itemize}
\end{thm}

\begin{rem}
 Theorem \ref{th1} removes   the condition  \(\rho(f)\leq 1\) in assertion \rm{(ii)} of Theorem \ref{thB},
and finds out all possible forms of entire solution $f$ with  $\rho_2(f)<1$ when $\alpha z+\beta$ in \eqref{1.8} is replaced by a  polynomial $g(z)$ with $\deg g >1$. Therefore, Theorem \ref{th1} answers the parts related to Theorem \ref{thB} in Question \ref{Q1} and Question \ref{Q2}.
\end{rem}

\begin{rem}
 Obviously, Theorem \ref{th1} improves   the condition  \(\rho(f)<\infty\)  in Theorem \ref{thB} to $\rho_2(f)<1$.
\end{rem}

    \begin{proof}[Proof of Theorem \ref{th1}]From \eqref{2.1}, it is easy to see that $f$ is transcendental. By the using the \cite[Lemma 1]{Liu1} to \eqref{2.1}, we have

\begin{align}
\label{2.2}
L_1(f) = a_0 f + a_1 f'' + a_2 f_c = \frac{\omega_2 e^{g + p} - \omega_1 e^{g - p}}{\omega_2 - \omega_1},
L_2(f) = b_0 f + b_1 f''+ b_2 f_c = \frac{e^{g + p} - e^{g - p}}{\omega_2 - \omega_1},
\end{align}
where \(f_c\) denotes \(f(z + c)\), $\omega_1=-\omega+\sqrt{\omega^2 - 1}, \omega_2=-\omega-\sqrt{\omega^2 - 1},$ and \(p\) is an entire function.  By the  \cite[Theorem 5.1]{halburd2014} and $\rho_2(f)<1$, we get $\rho(p)<1$.

The following equations can be obtained through direct elimination calculation on the equations \eqref{2.2},
\begin{align}
\label{2.3}
-\mathcal{A}_1f''-\mathcal{A}_3f_c=\frac{(b_0\omega_2 - a_0)e^{g + p} + (a_0 - b_0\omega_1)e^{g - p}}{\omega_2 - \omega_1},
\end{align}

\begin{align}
\label{2.4}
\mathcal{A}_1f-\mathcal{A}_2f_c=\frac{(b_1\omega_2 - a_1)e^{g + p} + (a_1 - b_1\omega_1)e^{g - p}}{\omega_2 - \omega_1},
\end{align}

\begin{align}
\label{2.5}
\mathcal{A}_3f+\mathcal{A}_2f''=\frac{(b_2\omega_2 - a_2)e^{g + p} + (a_2 - b_2\omega_1)e^{g - p}}{\omega_2 - \omega_1}.
\end{align}

Differentiating the equation \eqref{2.4} twice yields
\begin{align}
\label{2.6}
\mathcal{A}_1f''-\mathcal{A}_2f''_c=\frac{(b_1\omega_2 - a_1)[(g'+p')^2+g''+p'']e^{g + p} + (a_1 - b_1\omega_1)[(g'-p')^2+g''-g'']e^{g - p}}{\omega_2 - \omega_1}.
\end{align}
Adding  \eqref{2.6} and the shift of \eqref{2.5} together gives
\begin{align}
\label{2.7}
  \mathcal{A}_3f_c+\mathcal{A}_1f''&=   \frac{(b_1\omega_2 - a_1)[(g'+p')^2+g''+p'']e^{g + p} + (a_1 - b_1\omega_1)[(g'-p')^2+g''-g'']e^{g - p}}{\omega_2 - \omega_1}\\ \nonumber
 &+ \frac{(b_2\omega_2 - a_2)e^{g_c + p_c} + (a_2 - b_2\omega_1)e^{g_c - p_c}}{\omega_2 - \omega_1}.
\end{align}
Adding  \eqref{2.3} and  \eqref{2.7} together gives
\begin{align}
\label{2.8}
B_1 e^{p} + B_2 e^{-p} + B_3 e^{p_c + g_c - g} + B_4 e^{-p_c + g_c - g} = 0,
\end{align}
where
$B_1 = b_0\omega_2 - a_0 + (b_1\omega_2 - a_1)[(g'+p')^2+g''+p''],
B_2 = a_0 - b_0\omega_1 + (a_1 - b_1\omega_1)[(g'-p')^2+g''-p''],
B_3 = b_2\omega_2 - a_2,
B_4 = a_2 - b_2\omega_1.$

In order to use \cite[Theorem 1.51]{yang2003} to \eqref{2.8}, we next  consider whether $|a_2|+|b_2|=0$.

\setcounter{case}{0}
	\begin{case}
   \rm{ $|a_2|+|b_2|=0$.
 It is easy to see  $\mathcal{A}_1\not=0$ from \(r(\mathcal{A}) = 2\)  and $  \mathcal{A}_2=\mathcal{A}_3=0$. Then combining \eqref{2.6} and \eqref{2.3}  implies
\begin{align}
\label{2.9}
    \{(b_1\omega_2 - a_1)[(g'+p')^2+g''+p'']+(b_0\omega_2-a_0)\}e^{2p}+(a_1 - b_1\omega_1)[(g'-p')^2+g''-p'']\\ \nonumber
    +(a_0 - b_0\omega_1)=0.
\end{align}

If $p$ is a constant, then \eqref{2.9} gives
\begin{align}
\label{2.10}
(b_1\omega_2-a_1)e^{2p}+(a_1-b_1\omega_1)=(b_0\omega_2-a_0)e^{2p}+(a_0-b_0\omega_1)=0.
\end{align}
If $b_1\omega_2-a_1=0$, then \eqref{2.10} gives  $a_1-b_1\omega_1=0$. Substituting these into the equation \eqref{2.4} leads to the conclusion that
$f$
 is a constant, which is impossible. Thus $b_1\omega_2-a_1\not=0$. Similarly $a_0-b_0\omega_2\not=0$. Thus, \eqref{2.10} gives $
\frac{a_1 - b_1\omega_1}{a_1 - b_1\omega_2} = \frac{a_0 - b_0\omega_1}{a_0 - b_0\omega_2} = e^{2p},$
which implies \((a_0b_1 - a_1b_0)(\omega_2 - \omega_1) = 0\), that contradicts \(\mathcal{A}_1\not=0\). Thus  $p$ is not a constant.

Now $p$ is non-constant, then  \eqref{2.9} gives
\begin{align}
\label{2.11}
(b_1\omega_2 - a_1)[(g'+p')^2+g''+p'']+(b_0\omega_2-a_0)=0,\\  \nonumber
(a_1 - b_1\omega_1)[(g'-p')^2+g''-p'']+(a_0 - b_0\omega_1)=0.
\end{align}
 If $b_1\omega_2-a_1=0$, then  $a_0-b_0\omega_2=0$, which gives  \(a_0b_1 - a_1b_0= 0\), that contradicts \(\mathcal{A}_1\not=0\). Thus $b_1\omega_2-a_1\not=0$. Similarly $a_1 - b_1\omega_1\not=0$.
From \eqref{2.11}, we get
\begin{align}
\label{2.12}
2(g')^2+2(p')^2+2g''=\frac{a_0 - b_0\omega_2}{b_1\omega_2 - a_1}+\frac{a_0 - b_0\omega_1}{b_1\omega_1 - a_1}, \\ \nonumber
4g'p'+2p''=\frac{a_0 - b_0\omega_2}{b_1\omega_2 - a_1}-\frac{a_0 - b_0\omega_1}{b_1\omega_1 - a_1}.
\end{align}
We claim $\deg g=1$. Otherwise, if $\deg g>1$, then
the first equation of  \eqref{2.12} implies $p$ is a polynomial and $\deg g=\deg p>1$. Combining this with the second equation of \eqref{2.12}, we get $\deg g'p'=\deg p''$, which is  impossible. Thus $\deg g=1$ and the first equation of  \eqref{2.12} implies $p$ is a polynomial of $\deg p=1$.

Since $p,g$ are polynomials and $\mathcal{A}_2=0$,  \eqref{2.4} implies $f$  is of  finite order.
By the  assertion $\rm{(i)}$ of Theorem \ref{thB}, we get
Theorem \eqref{2.1}-(1).
    }
    \end{case}
\begin{case}\rm{
$|a_2|+|b_2|\not=0$. Since $\omega_1\not=\omega_2$, thus  $B_3=0 $ and $B_4=0$ at the same time  is impossible.

First, let's consider the situation $p$ is a constant. We claim $\deg g=1$. Otherwise,  if $\deg g>1$, then  $\deg (g_c-g)>1$. Therefore, from \eqref{2.8}, we get
\begin{align*}
B_1e^{p}+B_2e^{-p}=0, B_3e^{p}+B_4e^{-p}=0.
\end{align*}
That implies $B_1B_4-B_2B_3=\{b_0\omega_2 - a_0 + (b_1\omega_2 - a_1)[(g')^2+g'']\}(a_2 - b_2\omega_1)-\{a_0 - b_0\omega_1 + (a_1 - b_1\omega_1)[(g')^2+g'']\}(b_2\omega_2 - a_2)=0$. By comparing the coefficients of the polynomials, we can obtain
\begin{align*}
    (b_0 \omega_2 - a_0)(a_2 - b_2 \omega_1) &= (a_0 - b_0 \omega_1)(b_2 \omega_2 - a_2), \\
        (b_1 \omega_2 - a_1)(a_2 - b_2 \omega_1) &= (a_1 - b_1 \omega_1)(b_2 \omega_2 - a_2).
\end{align*}
It implies  \(\mathcal{A}_2=\mathcal{A}_3=0\).  It is easy to see  $\mathcal{A}_1\not=0$  by \(r(\mathcal{A}) = 2\). This scenario is consistent with Case 1, and we can obtain that $p$ is not a constant, this contradicts our initial hypothesis. Thus $\deg g=1$.

Now $\deg g=1$. If $\mathcal{A}_2=0$, \eqref{2.4} implies $f$ is of finite order. By the  the  assertion $\rm{(i)}$ of  Theorem \ref{thB}, we get Theorem \ref{th1}-(1).

We next consider $\mathcal{A}_2\not=0$. Let  $ g=\alpha z+\beta$, where \(\alpha(\not=0), \beta\) are constants.
If $\mathcal{A}_3=0$, then \eqref{2.5} becomes
\begin{align*}
\mathcal{A}_2f''=\frac{(b_2\omega_2 - a_2)e^{\alpha z+\beta+ p} + (a_2 - b_2\omega_1)e^{\alpha z+\beta- p}}{\omega_2 - \omega_1}.
\end{align*}
 That is \begin{align*}
      f(z) = \frac{e^{\alpha z + \beta}}{\alpha^2 \mathcal{A}_2 (\omega_2 - \omega_1)} \left[ (b_2 \omega_2 - a_2) e^{p} + (a_2 - b_2 \omega_1) e^{-p} \right] + C_3 z + C_4,
 \end{align*}
where $C_3,C_4$ are constants. This is conclusion $(2)$ of Theorem \ref{th1}, in which $A=0,B=p$.

If $\mathcal{A}_3\not=0$, then \eqref{2.5} becomes
\begin{align}
\label{sub2.1}
\mathcal{A}_3f+\mathcal{A}_2f''=\frac{(b_2\omega_2 - a_2)e^{\alpha z+\beta + p} + (a_2 - b_2\omega_1)e^{\alpha z+\beta- p}}{\omega_2 - \omega_1}.
\end{align}
Solving this equation gives
$f(z) = C_1 e^{k z} + C_2 e^{-k z} + \frac{D e^{\beta}}{\mathcal{A}_2} \cdot
\begin{cases}
\displaystyle \frac{e^{\alpha z}}{\alpha^2 + k^2}, & \alpha \neq \pm k, \\
\displaystyle \frac{z e^{\alpha z}}{2\alpha}, & \alpha = \pm k,
\end{cases}$\\
where $k = \sqrt{-\frac{\mathcal{A}_3}{\mathcal{A}_2}}$, $D = \frac{1}{\omega_2 - \omega_1} \left[ (b_2 \omega_2 - a_2) e^{p} + (a_2 - b_2 \omega_1) e^{-p} \right]$,  $C_1,C_2$ are constants. This is conclusion $(3)$ of Theorem \ref{th1}, in which $A=0,B=p$.

Following the previous discussion, we now proceed to analyze the case in which  $p$ is not a constant.
It is easy to see $\rho(B_i)<1,i=1,2,3,4$, by using \cite[Theorem 1.51]{yang2003} to \eqref{2.8} yields that one of
$p-p_c-g_c+g,p+p_c-g_c+g,-p-p_c-g_c+g,-p+p_c-g_c+g$ must be a constant.

\setcounter{subsection}{2}
			\setcounter{subcase}{0}
			\renewcommand{\thesubcase}{\arabic{subsection}.\arabic{subcase}}

\begin{subcase}\rm{$p-p_c-g_c+g=b$ is a constant. Applying \cite[Theorem 3.1]{Lu} to $p+g-(p+g)_c=b$ gives $p+g$ is a polynomial of degree 1. Thus $\deg g= \deg p$.
\eqref{2.8} becomes
$B_1 e^b + B_3 + (B_2 e^b + B_4 e^{2p - 2p_c}) e^{-2p} = 0,$
which implies
\begin{align}
\label{2.13}
B_1 e^b + B_3 =0, B_2 e^b + B_4 e^{2p - 2p_c} =0.
\end{align}

If $\deg p=\deg g>1$, then $g'+p'$ is a constant yields
 $g'-p'$ is not a constant, otherwise $g'$ is a constant, which contradicts with $\deg g>1$. Therefore  \eqref{2.13} gives $B_2 = a_0 - b_0\omega_1 + (a_1 - b_1\omega_1)[(g'-p')^2+g''-p'']
=0,~B_4 = a_2 - b_2\omega_1=0$. This implies $ a_0 - b_0 \omega_1 = a_1 - b_1 \omega_1 = a_2 - b_2 \omega_1 = 0,$ that contradicts $r(\mathcal{A}) = 2$. Thus  $\deg p=\deg g=1$.

 Let $ g=\alpha z+\beta$ and $p=Az+B$, where \(\alpha(\not=0), A(\not=0), B, \beta\) are constants. If $\mathcal{A}_2=0$, \eqref{2.4} implies $f$ has finite order. By the  the assertion $\rm{(i)}$ of  Theorem \ref{thB}, we get Theorem \ref{th1}-(1).

We next consider $\mathcal{A}_2\not=0$.
 If $\mathcal{A}_3=0$, then \eqref{2.5} becomes
\begin{align}
\label{sub2.2}
\mathcal{A}_2f''=\frac{(b_2\omega_2 - a_2)e^{\alpha z+\beta+ Az+B} + (a_2 - b_2\omega_1)e^{\alpha z+\beta- (Az+B)}}{\omega_2 - \omega_1}.
\end{align}
 That is \begin{align*}
   f(z) = \frac{1}{\mathcal{A}_2 (\omega_2 - \omega_1)}\cdot
\begin{cases}
\left( A_1 e^{(\alpha + A)z} + A_2 e^{(\alpha - A)z} \right) + C_3 z + C_4, & \alpha \neq \pm A, \\
\left( D_1 e^{2A z} + D_2 z^2 \right) + C_3 z + C_4, & \alpha = A, \\
\left( C_1 z^2 + C_2 e^{-2A z} \right) + C_3 z + C_4, & \alpha = -A,
\end{cases}
 \end{align*}
where $C_3,C_4$ are constants and
\begin{align*}
    &A_1 = \dfrac{(b_2 \omega_2 - a_2) e^{\beta + B}}{(\alpha + A)^2} ,A_2 = \dfrac{(a_2 - b_2 \omega_1) e^{\beta - B}}{(\alpha - A)^2}, \\
&D_1 = \dfrac{(b_2 \omega_2 - a_2) e^{\beta + B}}{4A^2} ,D_2 = \dfrac{(a_2 - b_2 \omega_1) e^{\beta - B}}{2}, \\
&C_1 = \dfrac{(b_2 \omega_2 - a_2) e^{\beta + B}}{2},C_2 = \dfrac{(a_2 - b_2 \omega_1) e^{\beta - B}}{4A^2}.
\end{align*}
This is conclusion $(2)$ of Theorem \ref{th1}.

If $\mathcal{A}_3\not=0$, then \eqref{2.5} becomes
\begin{align}
\label{sub 2.16}
\mathcal{A}_3f+\mathcal{A}_2f''=\frac{(b_2\omega_2 - a_2)e^{\alpha z+\beta + Az+B} + (a_2 - b_2\omega_1)e^{\alpha z+\beta- (Az+B)}}{\omega_2 - \omega_1}.
\end{align}
Solving this equation gives
\begin{align*}
f(z) = C_1 e^{\mu z} + C_2 e^{-\mu z} +
\begin{cases}
P_1 e^{k_1 z} + P_2 e^{k_2 z}, & k_1 \neq \pm \mu, k_2 \neq \pm \mu, \\
Q_1 z e^{k_1 z} + P_2 e^{k_2 z}, & k_1 = \pm \mu, k_2 \neq \pm \mu, \\
P_1 e^{k_1 z} + Q_2 z e^{k_2 z}, & k_1 \neq \pm \mu, k_2 = \pm \mu,
\end{cases}
\end{align*}
where
\begin{align*}
&\mu = \sqrt{-\dfrac{\mathcal{A}_3}{\mathcal{A}_2}},k_1 = \alpha + A,k_2 = \alpha - A ,\\
&P_1 = \dfrac{(b_2 \omega_2 - a_2) e^{\beta + B}}{(\omega_2 - \omega_1)[\mathcal{A}_2 (\alpha + A)^2 + \mathcal{A}_3]} ,P_2 = \dfrac{(a_2 - b_2 \omega_1) e^{\beta - B}}{(\omega_2 - \omega_1)[\mathcal{A}_2 (\alpha - A)^2 + \mathcal{A}_3]}, \\
&Q_1 = \dfrac{(b_2 \omega_2 - a_2) e^{\beta + B}}{2\mathcal{A}_2 (\omega_2 - \omega_1)(\alpha + A)},Q_2 = \dfrac{(a_2 - b_2 \omega_1) e^{\beta - B}}{2\mathcal{A}_2 (\omega_2 - \omega_1)(\alpha - A)}.
\end{align*}
This is conclusion $(3)$ of Theorem \ref{th1}.
}
\end{subcase}
\begin{subcase}\rm{
 $p+p_c-g_c+g=b$ is a constant.
 Applying \cite[Theorem 3.1]{Lu} to $p+p_c=g_c-g+b$ gives $p$ is a polynomial. Thus,  $\deg p = \deg g-1$ and  \eqref{2.8} becomes
 $B_1 e^{b} + B_4 + B_2 e^{-2p+b} + B_3 e^{2p_c} = 0.$ Then, using \cite[Theorem 1.51]{yang2003} to this equation, which  implies $B_2=B_3=0$, $B_1 e^{b} + B_4=0$. Therefore,
$a_0 - b_0 \omega_1 = a_1 - b_1 \omega_1 = b_1 \omega_2 - a_1 = b_2 \omega_2 - a_2 = 0.$
This gives \(\mathcal{A}_2 = \mathcal{A}_1= 0\), so \(\mathcal{A}_3\neq 0\) by \(r(\mathcal{A}) = 2\). Subsisting this into \eqref{2.5} gives
\begin{align*}
    f(z)=\frac{ (a_2 - b_2\omega_1)e^{g - p}}{\mathcal{A}_3(\omega_2 - \omega_1)}.
\end{align*}
This is conclusion $\rm{(4)}$ of Theorem \ref{th1}.
}
\end{subcase}

\begin{subcase}
\rm{ $-p-p_c-g_c+g=b$ is a constant.
As the same as Subcase 2.3, we get $p$ is a polynomial and $\deg p= \deg g-1$. Then \eqref{2.8} becomes
 $B_1 e^{2p+b} + B_2e^b+B_3+B_4 e^{-2p_c} = 0.$
Using  \cite[Theorem 1.51]{yang2003} to this equation, which  implies $B_1=B_4=0$, $B_2 e^{b} + B_3=0$. Therefore, $a_2 - b_2 \omega_1 = a_1 - b_1 \omega_1 = b_0 \omega_2 - a_0 = b_1 \omega_2 - a_1 = 0.$
This gives \(\mathcal{A}_2 = \mathcal{A}_1= 0\), so \(\mathcal{A}_3\neq 0\) by \(r(\mathcal{A}) = 2\). Subsisting this into \eqref{2.5} gives
\begin{align*}
    f(z)=\frac{ (b_2\omega_2 - a_2)e^{g(z)+p(z)}}{\mathcal{A}_3(\omega_2 - \omega_1)}.
\end{align*}
This is conclusion $\rm{(4)}$ of Theorem \ref{th1}.
}
\end{subcase}

\begin{subcase}
\rm{ $-p+p_c-g_c+g=b$ is a constant. Applying \cite[Theorem 3.1]{Lu} to
$g-p-(g-p)_c=b$ gives $g-p$ is a polynomial. Thus, $\deg (g-p)=1$ and $p$ is a polynomial of $\deg g=\deg p$. The remaining proof is similar to Subcase 2.1, we get  the conclusion $(3)$ of Theorem \ref{th1}.

}
\end{subcase}

}
\end{case}
In summary, Theorem \ref{th1} is proved completely.
\end{proof}

\section{quadratic trinomial Fermat  type \texorpdfstring{$q$}{q}-difference differential equations }
\label{s3}
In this section, the existence of solutions to the Fermat  type  $q$-difference differential equation \((Q_1(f))^2 + 2\omega Q_1(f)L_2(f) + (Q_2(f))^2 = e^{g(z)}\) is studied, thereby solving the parts related to Theorem \ref{thC} in Question \ref{Q1} and Question \ref{Q2}.
\begin{thm}
	\label{th2}
	Assume that  $Q_1, Q_2$ are defined in \eqref{Q}, $q$ and $\omega $ are  non-zero constants such that \(q^n\neq \pm 1,\) $n\in \mathbb{N}^+$, \( \omega \neq \pm 1\) and $\omega_1=-\omega+\sqrt{\omega^2 - 1}, \omega_2=-\omega-\sqrt{\omega^2 - 1},$ $g$ is a non-constant polynomial, and \(r(\mathcal{A}) = 2\). If the equation
\begin{equation}
\label{3.1}
(Q_1(f))^2 + 2\omega Q_1(f)Q_2(f) + (Q_2(f))^2 = e^{2g(z)}
\end{equation}
admits an entire solution \(f\) of $\rho_2(f)<\infty$, then $ g=\alpha z+\beta$, where \(\alpha(\not=0), \beta\) are constants, and  one the following
cases holds.
	\begin{itemize}
		\item[\rm{(1)}]  \(\mathcal{A}_2= 0\), $f(z) = C_1 e^{(A + \alpha)z} + C_2 e^{(-A + \alpha) z}$,
        \(C_1=\frac{(b_2\omega_2 - a_2)e^{B+\beta}}{(\omega_2 - \omega_1)\mathcal{A}_3}\) and  \(C_2=\frac{(a_2 - b_2\omega_1)e^{-B+\beta}}{(\omega_2 - \omega_1)\mathcal{A}_3}\) or \(C_1=\frac{(b_1\omega_2 - a_1)e^{B+\beta}}{(\omega_2 - \omega_1)\mathcal{A}_1}\) and \(C_2=\frac{(a_1 - b_1\omega_1)e^{-B+\beta}}{(\omega_2 - \omega_1)\mathcal{A}_1}\), where $ A, B$ are arbitrary constants.
    \item [\rm{(2)}]\(\mathcal{A}_2\not= 0\) and \(\mathcal{A}_3= 0\),  $f(z) = \frac{(b_2 \omega_2 - a_2) e^{(\alpha + A)z + \beta + B}}{\mathcal{A}_2 (\omega_2 - \omega_1) (\alpha + A)^2} + C_1 z + C_2$ or $f(z) = \frac{(a_2 - b_2 \omega_1) e^{(\alpha - A)z + \beta - B}}{\mathcal{A}_2 (\omega_2 - \omega_1) (\alpha - A)^2} + C_1 z + C_2$,
where $A,B,C_1,C_2$ are constants and $\alpha-A=q(\alpha+A)$ or $q(\alpha-A)=\alpha+A$, respectively.

 \item [\rm{(3)}] \(\mathcal{A}_2\not= 0\) and \(\mathcal{A}_3\not= 0\),
\begin{align*}
f(z) = C_1 e^{\mu z} + C_2 e^{-\mu z} +
\begin{cases}
P_1 e^{k_1 z} + P_2 e^{k_2 z}, & k_1 \neq \pm \mu, k_2 \neq \pm \mu, \\
Q_1 z e^{k_1 z} + P_2 e^{k_2 z}, & k_1 = \pm \mu, k_2 \neq \pm \mu, \\
P_1 e^{k_1 z} + Q_2 z e^{k_2 z}, & k_1 \neq \pm \mu, k_2 = \pm \mu,
\end{cases}
\end{align*}
where  $C_1, C_2, A,B$ are constants, $\mu = \sqrt{-\dfrac{\mathcal{A}_3}{\mathcal{A}_2}},k_1 = \alpha + A,k_2 = \alpha - A $,
\begin{align*}
&P_1 = \dfrac{(b_2 \omega_2 - a_2) e^{\beta + B}}{(\omega_2 - \omega_1)[\mathcal{A}_2 (\alpha + A)^2 + \mathcal{A}_3]} ,P_2 = \dfrac{(a_2 - b_2 \omega_1) e^{\beta - B}}{(\omega_2 - \omega_1)[\mathcal{A}_2 (\alpha - A)^2 + \mathcal{A}_3]}, \\
&Q_1 = \dfrac{(b_2 \omega_2 - a_2) e^{\beta + B}}{2\mathcal{A}_2 (\omega_2 - \omega_1)(\alpha + A)},Q_2 = \dfrac{(a_2 - b_2 \omega_1) e^{\beta - B}}{2\mathcal{A}_2 (\omega_2 - \omega_1)(\alpha - A)}.
\end{align*}

	\end{itemize}
\end{thm}

\begin{rem}
Theorem \ref{th2} shows that the equation  \eqref{1.8} has no  entire solutions $f$ of $\rho_2(f)<\infty$
when $\alpha z+\beta$ in \eqref{1.8} is replaced by a  polynomial $g(z)$ with $\deg g >1$. And
 Theorem \ref{th2} removes   the condition  \(\rho(f)\leq 1\) in assertion
$\rm{(ii)}$ of Theorem \ref{thC}.  Therefore the parts related to Theorem \ref{thC} in Question \ref{Q1} and Question \ref{Q2} are solved.
\end{rem}

\begin{rem}

 Theorem \ref{th2} improves   the condition  \(\rho(f)<\infty\)  in Theorem \ref{thC}  to $\rho_2(f)<\infty$.
\end{rem}

The following example shows condition \(q^n\neq \pm 1,\) $n\in \mathbb{N}^+$, is necessary.
\begin{exa}\cite[Example 1.1]{Ga}
    \( f(z) = \pm \frac{\sqrt{6}}{6} e^{\frac{z^3}{2}} \) is a transcendental entire solution of
\[
\left(f  (-\frac{1}{2} + \frac{\sqrt{3}}{2}i ) z \right)^2 + 4 f(z) f \left( ( -\frac{1}{2} + \frac{\sqrt{3}}{2}i ) z \right) + f(z)^2 = e^{z^3}.
\]
Here, \( g(z) = z^3 \), \( q = -\frac{1}{2} + \frac{\sqrt{3}}{2}i \), $q^3=1$.

\end{exa}

\begin{proof}[Proof of Theorem \ref{th2}]
Similar to   the beginning of proof of Theorem \ref{th1}, we can get the  following three equations,
\begin{align}
\label{3.2}
-\mathcal{A}_1f''-\mathcal{A}_3f_q=\frac{(b_0\omega_2 - a_0)e^{g + p} + (a_0 - b_0\omega_1)e^{g - p}}{\omega_2 - \omega_1},
\end{align}

\begin{align}
\label{3.3}
\mathcal{A}_1f-\mathcal{A}_2f_q=\frac{(b_1\omega_2 - a_1)e^{g + p} + (a_1 - b_1\omega_1)e^{g - p}}{\omega_2 - \omega_1},
\end{align}

\begin{align}
\label{3.4}
\mathcal{A}_3f+\mathcal{A}_2f''=\frac{(b_2\omega_2 - a_2)e^{g + p} + (a_2 - b_2\omega_1)e^{g - p}}{\omega_2 - \omega_1},
\end{align}
 where $p$ is an entire function of $\rho(p)<\infty$, and  $f_q=f(qz)$.
Differentiating the equation \eqref{3.3} twice yields
\begin{align}
\label{3.5}
\mathcal{A}_1f''-q^2\mathcal{A}_2f''_q=\frac{(b_1\omega_2 - a_1)[(g'+p')^2+g''+p'']e^{g + p} + (a_1 - b_1\omega_1)[(g'-p')^2+g''-g'']e^{g - p}}{\omega_2 - \omega_1}.
\end{align}
Adding  \eqref{3.2} and  \eqref{3.5} together gives
\begin{align}
\label{3.6}
  -\mathcal{A}_3f_q-q^2\mathcal{A}_2f''_q&=\frac{(b_1\omega_2 - a_1)[(g'+p')^2+g''+p'']e^{g + p} + (a_1 - b_1\omega_1)[(g'-p')^2+g''-g'']e^{g - p}}{\omega_2 - \omega_1}\\ \nonumber
  &+\frac{(b_0\omega_2 - a_0)e^{g + p} + (a_0 - b_0\omega_1)e^{g - p}}{\omega_2 - \omega_1}.
\end{align}
Taking the $q$-shift of \eqref{3.4} and combining \eqref{3.6} gives
\begin{align}
\label{3.7}
(1-q^2)\mathcal{A}_2f''_q&=\frac{(b_2\omega_2 - a_2)e^{g_q + p_q} + (a_2 - b_2\omega_1)e^{g_q - p_q}}{\omega_2 - \omega_1}+\frac{(b_0\omega_2 - a_0)e^{g + p} + (a_0 - b_0\omega_1)e^{g - p}}{\omega_2 - \omega_1}\\ \nonumber
&+\frac{(b_1\omega_2 - a_1)[(g'+p')^2+g''+p'']e^{g + p} + (a_1 - b_1\omega_1)[(g'-p')^2+g''-g'']e^{g - p}}{\omega_2 - \omega_1},
\end{align}

\begin{align}
\label{3.8}
(q^2-1)\mathcal{A}_3f_q=&=\frac{(b_1\omega_2 - a_1)[(g'+p')^2+g''+p'']e^{g + p} + (a_1 - b_1\omega_1)[(g'-p')^2+g''-g'']e^{g - p}}{\omega_2 - \omega_1}\\ \nonumber
  &+\frac{(b_0\omega_2 - a_0)e^{g + p} + (a_0 - b_0\omega_1)e^{g - p}}{\omega_2 - \omega_1}+q^2\frac{(b_2\omega_2 - a_2)e^{g_q + p_q} + (a_2 - b_2\omega_1)e^{g_q - p_q}}{\omega_2 - \omega_1}.
\end{align}

If $|a_2|+|b_2|=0$, then  $\mathcal{A}_1\not=0$ by \(r(\mathcal{A}) = 2\)  and $  \mathcal{A}_2=\mathcal{A}_3=0$. Thus the remaining proof is similar to  Case 1 in the proof of Theorem \ref{th1}, we get  Theorem \ref{th2}-(1).

Next, we consider $|a_2|+|b_2|\not=0$. If  $\mathcal{A}_3=\mathcal{A}_2=0$, then  $|a_2|+|b_2|\not=0$  implies $\mathcal{A}_1=0$, which is a contradiction.
Thus, we discuss three cases  $\mathcal{A}_3=0, \mathcal{A}_2\not=0$, $\mathcal{A}_3\not=0, \mathcal{A}_2=0$ and  $\mathcal{A}_3\not=0, \mathcal{A}_2\not=0$.

\setcounter{case}{0}
\begin{case}
\rm{$\mathcal{A}_3=0, \mathcal{A}_2\not=0$. \eqref{3.8} reduces into
\begin{align}
\label{3.9}
B_1 e^{p} + B_2 e^{-p} + B_3 e^{p_q + g_q - g} + B_4 e^{-p_q + g_q - g} = 0,
\end{align}
where
$B_1 = b_0\omega_2 - a_0 + (b_1\omega_2 - a_1)[(g'+p')^2+g''+p''],
B_2 = a_0 - b_0\omega_1 + (a_1 - b_1\omega_1)[(g'-p')^2+g''-p''],
B_3 = q^2(b_2\omega_2 - a_2),
B_4 = q^2(a_2 - b_2\omega_1).$ We claim that $B_3=0 $ and $B_4=0$ at the same time is impossible. Otherwise, $\omega_1=\omega_2$ or $a_2=b_2=0$, which is a contradiction.

If \( p \) is a constant, then \eqref{3.9} becomes
\[
B_1 e^p + B_2 e^{-p} + \left[ B_3 e^p + B_4 e^{-p} \right] e^{g_q-g} = 0.
\]
Since \( g_q - g \) is non-constant, by \( q^n \neq 1 \), then
$B_1 e^p + B_2 e^{-p}=0$ and $ B_3 e^p + B_4 e^{-p} =0.$
That implies \( B_1 B_4 - B_2 B_3 = \{b_0\omega_2 - a_0 + (b_1\omega_2 - a_1)[(g')^2+g'']\}q^2(a_2 - b_2\omega_1)-\{a_0 - b_0\omega_1 + (a_1 - b_1\omega_1)[(g')^2+g'']\}q^2(b_2\omega_2 - a_2)=0 \).
 By comparing the coefficients of the polynomials, we can obtain
\begin{align*}
    (b_0 \omega_2 - a_0)(a_2 - b_2 \omega_1) &= (a_0 - b_0 \omega_1)(b_2 \omega_2 - a_2), \\
        (b_1 \omega_2 - a_1)(a_2 - b_2 \omega_1) &= (a_1 - b_1 \omega_1)(b_2 \omega_2 - a_2).
\end{align*}
This implies  \(\mathcal{A}_2=\mathcal{A}_3=0\), which  contradicts $\mathcal{A}_2\not=0$. Thus, \( p \) is non-constant.

Since $\rho(B_i)<\infty ,i=1,2,3,4$, using  \cite[Theorem 1.51]{yang2003} to \eqref{3.9} yields that one of
$p-p_q-g_q+g,p+p_q-g_q+g,-p-p_q-g_q+g,-p+p_q-g_q+g$ must be a polynomial.
Therefore, from \cite[Lemma 5]{Liu1}, we get $p$ is a polynomial and $\deg g= \deg p$.  Now $B_i,i=1,2,3,4,$ are polynomial, using  \cite[Theorem 1.51]{yang2003} to \eqref{3.9} again yields that one of
$p-p_q-g_q+g,p+p_q-g_q+g,-p-p_q-g_q+g,-p+p_q-g_q+g$ must be a constant.

\setcounter{subsection}{1}
			\setcounter{subcase}{0}
			\renewcommand{\thesubcase}{\arabic{subsection}.\arabic{subcase}}

\begin{subcase}
\rm{
     $p-p_q-g_q+g=b$ is a constant. Then \eqref{3.9} becomes
\begin{align*}
B_1e^b+B_3+B_2e^{-2p+b}+B_4e^{-2p_q}=0.
\end{align*}
By the second main theorem, we have $B_4=q^2(a_2 - b_2\omega_1)=0$. Subsisting this and $\mathcal{A}_3=0$ into \eqref{3.4} implies
$\mathcal{A}_2f''=\frac{(b_2\omega_2 - a_2)e^{g + p}}{\omega_2 - \omega_1}.$ Since $p-p_q-g_q+g=p+g-(p+g)_q=b$, then $p+g$ is a constant. Therefore, $\mathcal{A}_2f''=\frac{(b_2\omega_2 - a_2)e^{g + p}}{\omega_2 - \omega_1}$  is a constant. It contradicts that  $f$ is  transcendental.
}
\end{subcase}

\begin{subcase}
\rm{
 $p+p_q-g_q+g=b$ is a constant. Then \eqref{3.9} becomes
\begin{align*}
B_1e^b+B_2e^{b-2p}+B_3e^{2p_q}+B_4=0.
\end{align*}
Using \cite[Theorem 1.51]{yang2003} to this equation, which  implies $B_2=B_3=0$, $B_1 e^{b} + B_4=0$.   $p+p_q-g_q+g=b$ implies $g'+p'=q(g'-p')_q$.  Therefore, if $\deg (g'-p')\ge1$, then $B_2=B_3=0$ and  $B_1 e^{b} + B_4=0$. This  yields
$a_0 - b_0 \omega_1 = a_1 - b_1 \omega_1 = b_1 \omega_2 - a_1 = b_2 \omega_2 - a_2 = 0.$ This gives \(\mathcal{A}_2 = \mathcal{A}_1= 0\),  which contradicts $\mathcal{A}_2 \not=0. $ Thus $\deg (g'-p')=0$. By  $g'+p'=q(g'-p')_q$, we get  $g'+p'$ is also a constant. Therefore $\deg g=\deg p=1$. Let $ g=\alpha z+\beta$ and $p=Az+B$, where \(\alpha(\not=0), A(\not=0), B, \beta\) are constants. Subsisting $\mathcal{A}_3=0, g,p, B_3=q^2(b_2\omega_2-a_2)=0$ into \eqref{3.4} gives $\mathcal{A}_2f''=\frac{(a_2 - b_2\omega_1)e^{\alpha z+\beta - (Az+B)}}{\omega_2 - \omega_1}.$
It is easy to get
\begin{align*}
f(z) = \frac{(a_2 - b_2 \omega_1) e^{(\alpha - A)z + \beta - B}}{\mathcal{A}_2 (\omega_2 - \omega_1) (\alpha - A)^2} + C_1 z + C_2,
\end{align*}
where $C_1$, $C_2$ are constants. This is the conclusion $(2)$ of Theorem \ref{th2}.
}
\end{subcase}

\begin{subcase}
\rm{
 $-p-p_q-g_q+g=b$ is a constant. Then \eqref{3.9} becomes
\begin{align*}
B_1e^{b+2p}+B_2e^{b}+B_3+B_4e^{-2p_q}=0.
\end{align*}
Using  \cite[Theorem 1.51]{yang2003} to this equation, which  implies $B_1=B_4=0$, $B_2 e^{b} + B_3=0$.
$-p-p_q-g_q+g=b$ implies $g'-p'=q(g'+p')_q$.  Therefore, similar as the case $p+p_q-g_q+g$ is a constant, we can get $ g=\alpha z+\beta$ and $p=Az+B$, where \(\alpha(\not=0), A(\not=0), B, \beta\) are constants. Subsisting $\mathcal{A}_3=0, g,p, B_4=0$ into \eqref{3.4} gives
$\mathcal{A}_2f''=\frac{(b_2\omega_2 - a_2)e^{\alpha z+\beta +(Az+B)}}{\omega_2 - \omega_1}.$ Thus
\begin{align*}
f(z) = \frac{(b_2 \omega_2 - a_2) e^{(\alpha + A)z + \beta + B}}{\mathcal{A}_2 (\omega_2 - \omega_1) (\alpha + A)^2} + C_1 z + C_2,
\end{align*}
where $C_1$, $C_2$ are constants.
This is the conclusion $(2)$ of Theorem \ref{th2}.
}
\end{subcase}

\begin{subcase}
\rm{
 $-p+p_q-g_q+g=b$ is a constant. Then we get $g-p$ is a constant and then \eqref{3.9} becomes
\begin{align*}
B_1e^{b+2p}+B_2e^b+B_3e^{2p_q}+B_4=0.
\end{align*}
This  implies $B_1=B_3=0$, then using the same methods as  the case $p-p_q-g_q+g=b$ is a constant, we also get a contradiction. }
\end{subcase}
}
\end{case}

\begin{case}
\rm{$\mathcal{A}_3\not=0, \mathcal{A}_2=0$. \eqref{3.7} reduces into
\begin{align}
\label{3.10}
D_1 e^{p} + D_2 e^{-p} + D_3 e^{p_q + g_q - g} +D_4 e^{-p_q + g_q - g} = 0,
\end{align}
where
$D_1 = b_0\omega_2 - a_0 + (b_1\omega_2 - a_1)[(g'+p')^2+g''+p''],
D_2 = a_0 - b_0\omega_1 + (a_1 - b_1\omega_1)[(g'-p')^2+g''-p''],
D_3 = b_2\omega_2 - a_2,
D_4 = a_2 - b_2\omega_1.$  We claim  $D_3=0 $ and $D_4=0$  at the same time is impossible, otherwise $\omega_1=\omega_2$ or $a_2=b_2=0$, which is a contradiction. As the same as Case 1, we can get $p$ is no-constant.

Using \cite[Theorem 1.51]{yang2003} to \eqref{3.10} yields that one of
$p-p_q-g_q+g,p+p_q-g_q+g,-p-p_q-g_q+g,-p+p_q-g_q+g$ must be a polynomial. Therefore, from \cite[Lemma 5]{Liu1}, we get $p$ is a polynomial and $\deg g= \deg p$. Using \cite[Theorem 1.51]{yang2003} to \eqref{3.10} again yields
$p-p_q-g_q+g,p+p_q-g_q+g,-p-p_q-g_q+g,-p+p_q-g_q+g$ must be a constant.
Using the same proof method as in Case 1, we can obtain that
$\deg g =1$. By  $\mathcal{A}_2=0$, \eqref{3.4} gives
$f$ is of finite order.
Since  $\deg g=1$ and $\mathcal{A}_2=0$, by
 Theorem \ref{thC}- $\rm{(i)}$, we get Theorem \ref{th2}-(1).

}
\end{case}

\begin{case}
\rm{$\mathcal{A}_3\not=0, \mathcal{A}_2\not=0$. By \eqref{3.7} and \eqref{3.8}, we get

\begin{equation}
\label{3.11}
E_1 e^{p} + E_2 e^{-p} + E_3 e^{(p+g)_q-g} + E_4 e^{(g-p)_q-g} = 0,
\end{equation}
 where
\begin{align*}
E_1 &= \mathcal{A}_2 (b_1\omega_2 - a_1) [u'''' + 4u'u''' + 3(u'')^2 + 6(u')^2 u'' + (u')^4] \\
&\quad + [\mathcal{A}_2 (b_0\omega_2 - a_0) + \mathcal{A}_3 (b_1\omega_2 - a_1)] [u'' + (u')^2] + \mathcal{A}_3 (b_0\omega_2 - a_0), \\
E_2 &= \mathcal{A}_2 (a_1 - b_1\omega_1) [v'''' + 4v'v''' + 3(v'')^2 + 6(v')^2 v'' + (v')^4] \\
&\quad + [\mathcal{A}_2 (a_0 - b_0\omega_1) + \mathcal{A}_3 (a_1 - b_1\omega_1)] [v'' + (v')^2] + \mathcal{A}_3 (a_0 - b_0\omega_1), \\
E_3 &= \mathcal{A}_2 q^2 (b_2\omega_2 - a_2) [u_q'' + (u_q')^2] + \mathcal{A}_3 (b_2\omega_2 - a_2), \\
E_4 &= \mathcal{A}_2 q^2 (a_2 - b_2\omega_1) [v_q'' + (v_q')^2] + \mathcal{A}_3 (a_2 - b_2\omega_1), u = g + p,v = g - p.
\end{align*}

\setcounter{subsection}{3}
			\setcounter{subcase}{0}
			\renewcommand{\thesubcase}{\arabic{subsection}.\arabic{subcase}}

\begin{subcase}\rm{ $E_i\equiv0,i=1,2,3,4$. Since $\omega_1\not=\omega_2$, then   $a_2-b_2\omega_1=0 $ and $b_2\omega_2-a_2=0$  at the same time is impossible.

If $b_2\omega_2-a_2\not=0$ and $a_2-b_2\omega_1\not=0$, then  $E_3=0$ implies $u''_q+(u'_q)^2=c_1$, where $c_1$ is a non-zero constant. Thus $2T(r,u'_q)\le m(r,\frac{u''_q}{u'_q})+m(r,u'_q)+O(1)$.  Therefore, $u'_q$ is a polynomial. $u''_q+(u'_q)^2=c_1$  means $u'_q$ is a constant. Similar, $E_4=0$ implies $v'_q$  is a constant. Thus, $u'_q+v'_q=2g'$  and $u'_q-v'_q=2p'_q$ are constants.
$E_3=E_4=0$ gives $2p''_q+4p'_qg'_q=0$.  Combining this and $\deg g>0$, we get $p'_q=0$. Thus, $p$ is a constant. Let  $ g=\alpha z+\beta$, where \(\alpha(\not=0), \beta\) are constants. Subsisting $p,g$ into \eqref{3.4}, we also get \eqref{sub2.1}.
Solving this equation \eqref{sub2.1} gives
the conclusion $(3)$  of Theorem \ref{th1}, in which $A=0, p=B$.

If $b_2\omega_2-a_2=0$, then $a_2-b_2\omega_1\not=0$  by $\omega_1\not=\omega_2$. If  $a_2-b_2\omega_1=0$, then $b_2\omega_2-a_2\not=0$ by $\omega_1\not=\omega_2$. Without loss of generality, let's consider case $b_2\omega_2-a_2=0$ and $a_2-b_2\omega_1\not=0$.
From $\mathcal{A}_2\not=0$ and $b_2\omega_2-a_2=0$, we get $b_1\omega_2-a_1\not=0$. $E_4=0$  gives  $v''+(v')^2=\frac{-\mathcal{A}_3}{\mathcal{A}_2}$. Using \cite[Lemma 3.3]{hayman} to $v''+(v')^2=\frac{-\mathcal{A}_3}{\mathcal{A}_2}$  implies $v'$ is polynomial. Since $\deg v' >\deg v''$, then $v'$ is a constant.
By using \cite[Lemma 3.3]{hayman}  to $E_1=0$, we also  get $u'$ is a polynomial. Since $E_1=0$, by comparing the degrees of the polynomials, we can obtain that $u'$ is a constant. Therefore  $u'+v'=2g'$ and $u'-v'=2p'$ are constants. Let  $ g=\alpha z+\beta$ and $p=Az+B$, where \(\alpha(\not=0), A(\not=0), B, \beta\) are constants. Subsisting $p,g$ into \eqref{3.4}, we also get \eqref{sub 2.16}, in which $b_2\omega_2-a_2=0$. Solving this equation \eqref{sub 2.16} gives
the conclusion $(3)$  of Theorem \ref{th2}, in which $b_2\omega_2-a_2=0$.
 }
\end{subcase}

\begin{subcase}
\rm{At least one of $E_i$ is not zero. Then \eqref{3.11} gives at least
two of $E_i$ is not zero, $i=1,2,3,4$.

At first, we consider  the case $p$ is a constant.  \eqref{3.11} gives
\begin{align}
\label{3.12}
E_1e^p+E_2e^{-p}=0, E_3e^p+E_4e^{-p}=0.
\end{align}
If $g'$ is not constant,  then \eqref{3.12} yields
\begin{equation*}
(b_1 \omega_2 - a_1)e^p + (a_1 - b_1 \omega_1)e^{-p}=0, \quad (b_2 \omega_2 - a_2)e^p + (a_2 - b_2 \omega_1)e^{-p}=0.
\end{equation*}
This implies $\mathcal{A}_2=0$. It contradicts the assumption  $\mathcal{A}_2\not=0$. Thus, $\deg g=1$.  Let  $ g=\alpha z+\beta$, where \(\alpha(\not=0), \beta\) are constants. Similar to the previous proof of Subcase 3.1, we can also arrive at the conclusion $(2)$  of Theorem \ref{th2}, in which $A=0, p=B$.


Next, we consider $p$ is non-constant.  Using  \cite[Theorem 1.51]{yang2003} and  \cite[Lemma 5]{Liu1} to \eqref{3.11}  gives that one of
$p-p_q-g_q+g,p+p_q-g_q+g,-p-p_q-g_q+g,-p+p_q-g_q+g$ is a constant, $p$ is a polynomial and $\deg g= \deg p$.

If $p-p_q-g_q+g=b$ is a constant, then \eqref{3.11} becomes
\begin{align*}
E_1e^b+E_3+E_2e^{-2p+b}+E_4e^{-2p_q}=0.
\end{align*}
This gives $E_1e^b+E_3=E_2=E_4=0$.
If \( a_2 - b_2 \omega_1 \neq 0 \),  then $E_4=0$ leads that \( g' - p' \) is a constant. If \( a_2 - b_2 \omega_1 = 0 \), by \( a_2 b_1 - a_1 b_2 =\mathcal{A}_2\neq 0 \), then \( a_1 - b_1 \omega_1 \neq 0 \).  Thus \( E_2 \equiv 0 \) implies that \( g' - p' \) is a constant. Since $p-p_q-g_q+g=p+g-(p+g)_q=b$, then $p'+g'$ is also a constant. It means \( \deg p = \deg g = 1 \).  Let $ g=\alpha z+\beta$ and $p=Az+B$, where \(\alpha(\not=0), A(\not=0), B, \beta\) are constants. Subsisting $p,g$ into \eqref{3.4}, we also get \eqref{sub 2.16}. Solving this equation gives
the conclusion $(3)$  of Theorem \ref{th2}.

If $p+p_q-g_q+g=b$ is a constant, then \eqref{3.11} becomes
\begin{align*}
E_1e^b+E_2e^{b-2p}+E_3e^{2p_q}+E_4=0.
\end{align*}
Using  \cite[Theorem 1.51]{yang2003} to this equation  implies $E_2=E_3=0$, $E_1 e^{b} + E_4=0$.   $p+p_q-g_q+g=b$ implies $g'+p'=q(g'-p')_q$.
If \( a_2 - b_2 \omega_2 \neq 0 \),  then $E_3=0$ leads that \( g' + p' \) is a constant. Thus, $g'+p'=q(g'-p')_q$ implies $g'-p'$ is also a constant.  If \( a_2 - b_2 \omega_2 = 0 \), then \( a_1 - b_1 \omega_2 \neq 0 \) by
\( a_2 b_1 - a_1 b_2 =\mathcal{A}_2\neq 0 \).  Thus \(E_1 e^{b} + E_4= 0 \) implies that \( g' - p' \) is a constant. $p+p_q-g_q+g=b$ implies $g'+p'=q(g'-p')_q$ is also a constant.
In summary \( \deg p = \deg g = 1 \). Let $ g=\alpha z+\beta$ and $p=Az+B$, where \(\alpha(\not=0), A(\not=0), B, \beta\) are constants. Subsisting $p,g$ into \eqref{3.4}, we also get \eqref{sub 2.16}. Solving this equation gives
the conclusion $(3)$  of Theorem \ref{th2}.

The remaining two cases can be handled using similar methods as the cases $p-p_q-g_q+g$ is a constant and $p+p_q-g_q+g$ is a constant, and we still arrive at the conclusion $(3)$  of Theorem \ref{th2}; here, we omit the proof.
}

In summary,  Theorem \ref{th2} is proved completely.

\end{subcase}

}

\end{case}

\end{proof}

\section{Examples}

Some examples are provided in  this section
to illustrate the results from Theorem \ref{th1} and Theorem \ref{th2}.

\begin{exa}Let $c=2\pi i$, $\omega=2$ and
$\mathcal A = \begin{pmatrix}1 & 0& 0\\ 0 & 0 & 1 \end{pmatrix},$
$g(z) = z + \frac{\ln 6}{2}.$ Then, $\mathcal{A}_2 = 0$,   $f(z) = e^{z}$ is a solution of \eqref{2.1}. This  demonstrates  that  conclusion (1) of Theorem \ref{th1}  may  occur.

\end{exa}
\begin{exa}
 $f(z) = e^z$ is a solution of $(f''(z)+f(z+z\pi i))^2+4(f''(z)+f(z+2\pi i))(f''(z)-f(z+2\pi i))+(f''(z)-f(z+2\pi i))^2=4e^{2z}$.
 There  Matrix $\mathcal{A} = \begin{pmatrix} 0 & 1 & 1 \\ 0 & 1 & -1\end{pmatrix}$,
 $  r(\mathcal{A}) = 2, ~ \mathcal{A}_1 = 0,~\mathcal{A}_2 = -2 \neq 0, \mathcal{A}_3 = 0$. This  shows  that  conclusion (2) of Theorem \ref{th1} may  occur.
\end{exa}

\begin{exa}
 $f(z)=ze^z$ is a solution of
$(f(z)-(f(z+2\pi i))^2+\frac{\pi}{i}(f(z)-(f(z+2\pi))(f(z)-f''(z))+(f(z)-f''(z))^2=e^{2z+\ln4}$.
There Matrix $
    \mathcal{A} = \begin{pmatrix} 1 & 0 & -1 \\ 1& -1 & 0 \end{pmatrix},
    \mathcal{A}_2 =-1, \quad \mathcal{A}_3 = 1$. This  demonstrates that   conclusion (3) of Theorem \ref{th1} may  occur.
\end{exa}

Example \ref{ex1} illustrates  that  conclusion (4) of Theorem \ref{th1} may  occur.

\begin{exa}
    Let
$\mathcal{A} = \begin{pmatrix} 1 & 0& 0 \\ 0&1 & 0 \end{pmatrix}, \omega=2, g=e^z$, then $\mathcal{A}_2 = 0$ and $f=\frac{1}{\sqrt{6}}e^z$ is a solution of \eqref{3.1}. This  demonstrates  that  conclusion (1) of Theorem \ref{th2} may  occur.
\end{exa}

\begin{exa}Let $\omega = 2$, $g(z) = 2qz + \ln \sqrt{13}$ and
\[
Q_1(f) = f(z) - \frac{1}{4} f'(z) + 2f(qz), \quad Q_2(f) = f(z) - \frac{1}{4} f'(z) + f(qz),
\]
where $q$ is a non-zero constant. Then $\mathcal{A}_2 = \frac{1}{4}$, $\mathcal{A}_3 = -1$ and \eqref{3.1} has an entire solution $f(z) = e^{2z}$. This  demonstrates   that  conclusion (3) of Theorem \ref{th2} may  occur.
\end{exa}

	\section*{Declarations}
	\begin{itemize}
		\item \noindent{\bf Funding}
		This research work was supported by Graduate Research Fund Project of Guizhou Province ((2025YJSKYJJ107)) and the National Natural Science Foundation of China (Grant No. 12261023, 11861023).
		
		\item \noindent{\bf Conflicts of Interest}
		The authors declare that there are no conflicts of interest regarding the publication of this paper.

        \item\noindent{\bf Author Contributions}
All authors contributed to the study conception and design. All authors read and approved the final manuscript.

	\end{itemize}


\begin{thebibliography}{99}
	\normalsize
	\baselineskip=17pt
		
	\bibitem{B}Baker, I.N.: On a class of meromorphic functions. Proc. Amer. Math. Soc. 17, 819-822 (1966)
		
		
	\bibitem{Chen} Chen, W., Han, Q., Liu, J.B.: On Fermat Diophantine functional equations, little Picard theorem and beyond. Aequat. Math. 93, 425-432 (2019)

    \bibitem{Ga}Gao, Z.G.,  Gao, L.Y.,  Liu, M.L.: Entire solutions of two certain types of quadratic trinomial $q$-difference differential equations. AIMS Math. 8, 27659-27669 (2023)

 \bibitem{G2}Gundersen, G.G.: Meromorphic solutions of $f^6+g^6+h^6=1$. Analysis (Munich)
18, 285-290 (1998)


	\bibitem{G1}Gundersen, G.G.: Meromorphic solutions of $f^5+g^5+h^5=1$. Complex Var. Theory Appl.
 43, 293-298 (2001).

 \bibitem{G3}Gundersen, G.G., Tohge, K.: Entire and meromorphic solutions of \(f^{5}+g^{5}+h^{5}=1\). In: Symposium on Complex Differential and Functional Equations. Univ. Joensuu Dept. Math. Rep. Ser. 6, 57-67 (2004)

		\bibitem{Gong}Gong, Y., Yang, Q.: On entire solutions of two Fermat-type differential-difference equations. Bull. Iran. Math. Soc. 51, 17 (2025). https://doi.org/10.1007/s41980-024-00942-4
	\bibitem{Gross1}Gross, F.: On the equation $f^n+g^n=1$. Bull. Amer. Math. Soc. 72, 86-88 (1966)
		\bibitem{Gross2} Gross, F.: On the functional equation $f^n+g^n=h^n$. Amer. Math. Monthly  73, 1093-1096 (1966)

		
		
		
		\bibitem{halburd2014}Halburd, R.G.,   Korhonen, R.J.,  Tohge, K.: Holomorphic curves with shift-invariant hyperplane preimages.  Trans. Amer. Math. Soc. 366, 4267-4298 (2014)
		\bibitem{Haldar} Haldar, G., Ahamed, M.B.: Entire solutions of several quadratic binomial and trinomial partial differential-difference equations in $\mathbb C^2$. Anal. Math. Phys. 113, 1-38 (2022)
		
		\bibitem{hayman}Hayman, W.K.: Meromorphic Functions. Clarendon Press, Oxford (1964)

        \bibitem{h1}Hayman, W.K.: Waring's problem f\"ur analytische funktionen. Bayer. Akad. Wiss. Math.-Natur. Kl. Sitzungsber. 1-13 (1985)


        \bibitem{Hu} Hu, P.C., Wang, Q.: On meromorphic solutions of functional equations of Fermat type. Bull. Malays. Math. Sci. Soc. 42, 2497-2515 (2019)
		
		\bibitem{Lc}Liu, K., Cao, T.B., Cao, H.Z.: Entire solutions of Fermat type differential-difference equations. Arch. Math. 99, 147-155 (2012)

		\bibitem{Liu} Liu, K., Yang, L.Z.: A note on meromorphic solutions of Fermat type equations. An. Stiint. Univ. Al. I. Cuza Lasi Mat. (NS) 1, 317-325 (2016)

        \bibitem{Liu1}Liu, X.Y., L\"{u}, F., L\"{u}, W.R., Wang, J.: On entire solutions of two types of differential-difference Equations. Comput. Methods Funct. Theory (2025). https://doi.org/10.1007/s40315-025-00588-1



  \bibitem{Lu}  L\"u, F., L\"u, W.R., Li, C.P., Xu, J.F.: Growth and uniqueness related to complex differential and difference equations. Results Math. 74, Art. 30, 18 pp (2019)

  \bibitem{tai} Tai, Z.Y., Long, J.R., Xiang, X.X.: On entire solutions for several systems of
 quadratic trinomial Fermat type functional equations in $\mathbb{C}^2$. Rev. Real Acad. Cienc. Exactas, F\'is. Nat. Ser. A-Mat. 156, 1-20 (2024)
		
		\bibitem{Montel}Montel, P.: Le\'{c}ons sur les r\'ecurrences et leurs applications. Gauthier-Villar, Paris, (1927)
		
		\bibitem{S}Saleeby, E.G.: On complex analytic solutions of certain trinomial functional and partial differential equations. Aequat. Math.  85, 553-562 (2013)

\bibitem{Toda}Toda,  N.: On the functional equation $\sum_{i = 0}^{p} a_i f_i^{n_i}$. T\^{o}hoku Math. J. 23, 289-299 (1971)



        \bibitem{Wang}Wang, Q., Chen, W., Hu, P.C.:
        entire solutions of two certain Fermat-type differential-difference equations. Bull. Malays. Math. Sci. Soc. 43, 2951-2965 (2020)
		
		
		\bibitem{Xu}Xu, H.Y., Jiang, Y.Y.: Results on entire and meromorphic solutions for several systems of quadratic trinomial functional equations with two complex variables. Rev. Real Acad. Cienc. Exactas, F\'{\i}s. Nat. Ser. A-Mat. 116, 1-19 (2022)
        \bibitem{Yang1970}Yang, C.C.: A generalization of a theorem of P. Montel on entire functions. Proc. Amer. Math. Soc. 26, 332-324  (1970)
	\bibitem{yang2004}Yang, C.C., Li, P.: On the transcendental solutions of a certain type of nonlinear differential equations. Arch. Math. 82, 442-448 (2004)
		
		\bibitem{yang2003}Yang, C.C., Yi, H.X.: Uniqueness Theory of Meromorphic Functions. Kluwer Academic Publishers Group, Dordrecht (2003)
        \bibitem{Yl}Yang, L.Z., Zhang, J.L.: Non-existence of meromorphic solutions of a Fermat type functional equation. Aequat. Math. 76, 140-150 (2008)

		
		
	\end{thebibliography}
\end{document}